\documentclass[11pt]{amsart}
\usepackage{setspace}
\usepackage[T1]{fontenc}
\usepackage{amssymb,amsmath}
\usepackage{epsfig}
\usepackage{graphicx}
\usepackage{mathtools}
\usepackage{nicefrac}   
\usepackage{lipsum}
\usepackage{amscd}
\usepackage{tikz-cd}
\usepackage[margin=1.25in]{geometry}
\input xy
\xyoption{all}
\theoremstyle{plain}
\newtheorem{thm}{Theorem}[section]

\newtheorem{lem}[thm]{Lemma}

\theoremstyle{definition} \theoremstyle{definition}

\theoremstyle{remark}

\def\R{{\mathbb R}}
\def\Q{{\mathbb Q}}
\def\Z{{\mathbb Z}}
\def\C{{\mathbb C}}

\def\N{{\mathbb N}}
\def\X{{\sf X}}

\def\Br{{\rm Br}}
\def\Hom{{\rm Hom}}
\def\cont{{\rm cont}}

\title{The Algebraic Brauer Group of a Reductive Group over a Nonarchimedean Local Field}

\author{Dylon Chow}

\begin{document}

\begin{abstract}
    We show that for nonarchimedean local fields $F$, the pairing from the algebraic part of the Brauer group of a reductive group $G$ characterizes all continuous homomorphisms from $G(F)$ into $\Q/\Z$. This generalizes results of Loughran and Loughran-Tanimoto-Takloo-Bighash.
\end{abstract}

\maketitle

\section*{Introduction}

Let $F$ be a non-Archimedean local field of characteristic 0 and $X$ an algebraic variety defined over $F$. The set $X(F)$ of $F$-rational points on $X$ acquires a natural analytic topology from $F$. Each element of the Brauer group $\text{Br}(X)$ of $X$ defines a locally constant map from $X(F)$ into $\Q/\Z$.

Let $X$ be a $F$-variety and let $\Br X$ denote the Brauer group of $X$. If $L$ is a $k$-algebra and $x \in X(L)$, then $x:\text{Spec} \ L \rightarrow X$ induces a homomorphism $\Br \ X \rightarrow \Br \ L$. By composition with the invariant map $\text{Br}(F) \rightarrow \Q/\Z$ of local class field theory, each element $x \in X(F)$ defines a homomorphism $\Br \ X \rightarrow \Q/\Z$. Similarly, each element of $\Br \ X$ defines a map $X(F) \rightarrow \Q/\Z$.

Let $F$ be any field and let $\overline{F}$ be a separable closure of $F$. Let $\text{Br}_1  X$ be the kernel of the homomorphism $\text{Br} \ X \rightarrow \text{Br} \ X_{\overline{F}}$, and let $\text{Br}_0  X$ denote the image of $\text{Br} \ F \rightarrow \text{Br} X$. The \textit{algebraic part} of $\Br (X)$ is defined to be the quotient $\text{Br}_a X=\text{Br}_1 X / \text{Br}_0  X$. For an algebraic group $G$ over $F$, the morphism $e:\text{Spec} \ F \rightarrow G$ associated with the identity element $e \in G(F)$ induces a homomorphism $\text{Br} \ G \rightarrow \text{Br} \ F$. Let $\text{Br}_e G$ be the intersection of $\text{Br}_1 G$ with the kernel of $\text{Br} \ G \rightarrow \text{Br} \ F$. The quotient homomorphism $\text{Br}_1 G \to \text{Br}_a G$ restricts to an isomorphism $\text{Br}_e G \cong \text{Br}_a G$. Elements of $\text{Br}_e G$ define continuous homomorphisms from $G(F)$ into $\text{Br}(F)$ \cite[Lemme 6.9]{Sansuc1981}.

We consider the algebraic part $\text{Br}_a (G)$ of the Brauer group of a connected reductive $F$-group $G$. To state our main result, let $G^{sc}$ be the simply connected cover of the derived group $G^{der}$ of $G$ and let \[\rho:G^{sc} \to G\] be the natural morphism $G^{sc} \to G^{der} \hookrightarrow G$. The main purpose of this paper is to prove the following: 

\begin{thm} Let $F$ be a non-archimedean local field of characteristic $0$. Let $G$ be a connected reductive group defined over $F$ and let $\rho:G^{sc} \to G$ be the natural map. The pairing \[\Br_e G \times G(F) \to \Q/\Z\] induces an isomorphism \[\Br_e G \cong  \Hom_\cont(G(F)/\rho(G^{sc}(F)),\Q/\Z).\] 
\end{thm}

This generalizes a theorem of Loughran \cite{Loughran2018}, who proved it for tori, and Loughran, Takloo-Bighash, and Tanimoto \cite{LTBT2020}, who proved it for semisimple groups. We will address the situation of a number field in future work.

\section{Notation and conventions}

\subsection{\ } We use $F$ to denote a field. Let $\overline{F}$ be an algebraic closure of $F$ and write $F^s$ for the separable closure of $F$ in $\overline{F}$. We let $\Gamma=\Gamma_F$ denote the Galois group of $F^s$ over $F$; it is a profinite topological group with the Krull topology.

\subsection{\ } If $G$ is a connected reductive group defined over a field $F$ and $K$ is a field extension of $F$, we write $G_K$ for the $K$-group obtained from $G$ by extension of scalars. Let $\mathbb{G}_m$ be the multiplicative group scheme $\text{GL}_1$. For an algebraic group $G$, let $\X(G)$ denote the group of characters of $G$, i.e. the group of algebraic group homomorphisms $G \rightarrow \mathbb{G}_m$. We let $\X^*(G)=\X(G_{F^s})$. In other words, $\X^*(G)$ consists of the characters of $G$ defined over $F^s$. The group $\Gamma$ acts continuously on $\X^*(G)$.

\subsection{\ } If $A$ is an abelian group with the discrete topology on which a profinite group $\Gamma$ acts as a group of automorphisms, then $A$ is called a $\Gamma$-module if the action map $\Gamma \times A \rightarrow A, (\sigma, a) \mapsto \sigma a$ is continuous. Equivalently, $A$ is a $\Gamma$-module if for all $a \in A$ the stabilizer $\{\sigma \in \Gamma|\sigma a = a\}$ of $a$ is open in $\Gamma$.

\subsection{\ } Let $k$ be a field, $k_s$ a separable closure of $k$ and $G$ an algebraic $k$-group. Then $H^i(k,H)$ denotes the $i$-th cohomology set of the Galois group $\text{Gal}(k_s/k)$ of $k_s$ over $k$, with coefficients in $H(k_s)$ $(i=0,1)$ and, if $G$ is commutative, the $i$-th cohomology group of $\text{Gal}(k_s/k)$ in $G(k_s)$ for all $i \in \N$.

\subsection{\ } If $F'$ is a finite field extension of $F$ and $G$ is an algebraic group over $F'$, the Weil restriction of $G$ is the algebraic group $G_{F'/k}$ over $k$ such that for all $k$-algebras $R$, $G_{F'/F}(R)=G(F' \otimes R)$. By an induced $\Gamma$-module we mean a $\Gamma$-module that has a finite $\Gamma$-stable $\mathbb{Z}$-basis. We say than an $F$-torus $T$ is induced if $\X^*(T)$ is an induced $\Gamma$-module. Equivalently, an $F$-torus $T$ is induced if it is a finite product of tori of the form $(\mathbb{G}_m)_{k'/F}$ with $k'$ a finite separable extension of $F$.

\subsection{\ } As usual, $\Q, \R,$ and $\C$ will denote respectively the fields of rational, real, and complex numbers; $\Z$ denotes the ring of rational integers.

\subsection{\ } Sometimes our characters have values in $\Q/\Z$, in which case we use the exponential mapping $x \mapsto \text{exp}(2\pi i x)$ from $\Q/\Z$ to $\C^\times$ to view them as complex-valued characters.

\subsection{\ } Let $G_{der}$ denote the derived group of $G$, $G_{sc}$ the simply connected cover of $G_{der}$, and $G_{ad}$ the adjoint group of $G$, i.e., $G_{ad}=G/Z_G$ where $Z_G$ is the center of $G$. Let $\rho:G^{sc} \to G$ be the natural morphism. Given a maximal $F$-torus $T$ of $G$, let $T_{sc}=\rho^{-1}(T)$.

\section{Preliminaries}

\subsection{\ } Let $F$ be a local field or a number field and let $X$ be a $F$-variety. Let $\Br X$ be the Brauer group of $X$. If $L$ is a $k$-algebra and $x \in X(L)$, then $x:\text{Spec} \ L \rightarrow X$ induces a homomorphism $\Br \ X \rightarrow \Br \ L$. By composition with the invariant map $\text{Br}(F) \rightarrow \Q/\Z$ of local class field theory, each element $x \in X(F)$ defines a homomorphism $\Br \ X \rightarrow \Q/\Z$. Similarly, each element of $\Br \ X$ defines a map $X(F) \rightarrow \Q/\Z$.

\subsection{\ } Let $F$ be any field and let $\overline{F}$ be a separable closure of $F$. Let $\text{Br}_1  X$ be the kernel of the homomorphism $\text{Br} \ X \rightarrow \text{Br} \ X_{\overline{F}}$, and let $\text{Br}_0  X$ denote the image of $\text{Br} \ F \rightarrow \text{Br} X$. The \textit{algebraic part} of $\Br (X)$ is defined to be the quotient $\text{Br}_a X=\text{Br}_1 X / \text{Br}_0  X$. For an algebraic group $G$ over $F$, the morphism $e:\text{Spec} \ F \rightarrow G$ associated with the identity element $e \in G(F)$ induces a homomorphism $\text{Br} \ G \rightarrow \text{Br} \ F$. Let $\text{Br}_e G$ be the intersection of $\text{Br}_1 G$ with the kernel of $\text{Br} \ G \rightarrow \text{Br} \ F$. The quotient homomorphism $\text{Br}_1 G \to \text{Br}_a G$ restricts to an isomorphism $\text{Br}_e G \cong \text{Br}_a G$. Elements of $\text{Br}_e G$ define continuous homomorphisms from $G(F)$ into $\text{Br}(F)$ \cite[Lemme 6.9]{Sansuc1981}.

\subsection{\ } A morphism $f:T \rightarrow U$ of tori is defined over $F$ is a crossed module in a natural way and we can consider its cohomology groups $H^i(F,T \rightarrow U)$. We refer the reader to \cite{Borovoi1998} for the definition and properties of crossed modules and their cohomology.

\section{Non-archimedean local fields}

We prove the theorem in two stages. In the first stage we start from the case of tori and generalize the result for only those $G$ whose derived group is simply connected. 

\subsection{Tori}

\begin{lem}
Let $T$ be a torus over a local field $F$ of characteristic $0$. The bilinear pairing \[\text{Br}_e T \times T(F) \to \text{Br} \ F \subset \Q/\Z\] is perfect, i.e., the induced map \[\text{Br}_e T \to \text{Hom}(T(F),\Q/\Z)\] is an isomorphism of abelian groups.
\end{lem}

\begin{proof}
See \cite[Theorem 4.3]{Loughran2018}.
\end{proof}

\subsection{Groups with simply connected derived group}

Let $F$ be a $p$-adic field. Now assume that $G$ is such that $G^{der}=G^{sc}$. Define $T=G/G^{der}$. We have an exact sequence \[1 \to G^{der} \to G \to T \to 1.\] We get an exact sequence \[1 \to G^{der}(F) \to G(F) \to T(F) \to 1,\] and thus an isomorphism \[G(F)/j(G^{der}(F)) \cong T(F).\]

Since $\text{Pic}(\overline{G})=0$, we have (\cite[Lemme 6.9 (i)]{Sansuc1981}) canonical isomorphisms $H^2(F,\X^*(G)) \cong \Br_a G$ and $H^2(F,\X^*(T)) \cong \Br_a T$. The projection $G \to T$ yields a commutative diagram 

\begin{tikzcd}
\Br_a T \arrow{r}\arrow{d} 
& \Br_a G \arrow{d} \\
H^2(F,\X^*(T)) \arrow{r}& H^2(F,\X^*(G)).
\end{tikzcd}

The vertical arrows are isomorphisms. The first row is part of a long exact sequence coming from the exact sequence $1 \to G^{sc} \to G \to T \to 1$: (\cite[Corollaire 6.11]{Sansuc1981}) \[...\to Pic(G^{sc}) \to \Br_a T \to \Br_a G \to \Br_a G^{sc}.\] Since $Pic(G^{sc})=1$ and $\Br_a(G^{sc})=1$ (\cite[Lemme 9.4 (iv)]{Sansuc1981}), the horizontal arrow is an isomorphism. We get a commutative diagram 

\begin{tikzcd}
H^2(F,\X^*(T)) \arrow{r}\arrow{d}
& H^2(F,\X^*(G)) \arrow{d} \\
Hom(G(F)/j(G^{der}(F)),\Q/\Z) \arrow{r}& Hom(T(F),\Q/\Z).
\end{tikzcd}

This proves the result for groups whose derived group is simply connected.

\subsection{General reductive groups}

In the second stage we use the following result, which allows one to reduce to the case where the derived group is simply connected.

\begin{lem}
For any connected reductive $F$-group $G$ split by $K$, there exists an extension \[1 \to Z \to \widetilde{G} \to G \to 1\] such that

\begin{itemize}
    \item $Z$ is a central torus in $\widetilde{G}$,
    
    \item $Z$ is obtained from Weil restriction of scalars from a split $K$-torus, and
    
    \item $\widetilde{G}^{der}$ is simply connected.
\end{itemize}
\end{lem}

Such an extension is called a $z$-extension. We proceed with the proof of the general case. A similar result appears in \cite[Lemma A.1, Appendix]{LM2015}. Consider a $z$-extension as above. We get two more exact sequences. First, since $\text{Pic} Z=0$, we get from \cite[Corollary 6.11]{Sansuc1981} an exact sequence of abelian groups \[1 \to \Br_e G \to \Br_e \widetilde{G} \to \Br_e Z.\] Since $Z$ is an induced torus, $H^1(F,Z)=0$, and so we get another exact sequence \[1 \to Z(F) \to \widetilde{G}(F) \to G(F) \to 1.\]

Applying $\Hom(-,\Q/\Z)$ we get an exact sequence \[1 \to \Hom(G(F),\Q/\Z) \to \Hom(\widetilde{G}(F),\Q/\Z) \to \Hom(Z(F),\Q/\Z) \to 1.\] This induces an exact sequence \[1 \to \Hom(G(F)/G^{der}(F),\Q/\Z) \to \Hom(\widetilde{G}(F)/\widetilde{G}^{der}(F),\Q/\Z) \to \Hom(Z(F),\Q/\Z).\]

We get the following commutative diagram with exact rows:

\begin{tikzcd}
1 \arrow{r}
& \Br_e G \arrow{r}
& \Br_e \widetilde{G} \arrow{r}\arrow{d}
& \Br_e Z \arrow{d}\\
1 \arrow{r}&\Hom(G(F)/G^{der}(F),\Q/\Z)\arrow{r}&\Hom(\widetilde{G}(F)/\widetilde{G}^{der}(F),\Q/\Z)\arrow{r}&\Hom(Z(F),\Q/\Z)
\end{tikzcd}

The two vertical arrows are the isomorphisms constructed above. We define a homomorphism $\Br_e G \to \Hom(G(F)/G^{der}(F),\Q/\Z)$ to be the unique homomorphism that makes the diagram commute -- it is an isomorphism, which can be seen to be induced from the Brauer pairing. This completes the proof.

\section*{Acknowledgements}

The author thanks Dan Loughran and Ramin Takloo-Bighash for helpful comments.

\bibliographystyle{alpha}
\bibliography{biblio.bib}
\end{document}